\documentclass[11pt,twoside,leqno]{article}
\usepackage{amsmath,amsthm}
\usepackage{amssymb,latexsym}
\usepackage{enumerate}

\newcommand{\FF}{{\cal F}}
\newcommand{\LL}{{\cal L}}
\newcommand{\BL}{{\mathbb L}}
\newcommand{\BN}{{\mathbb N}}
\newcommand{\BR}{{\mathbb R}}

\newcommand{\fch}{\mbox{\rm\bf 1}}

\newtheorem{theorem}{Theorem}[section]
\newtheorem{proposition}[theorem]{Proposition}
\newtheorem{corollary}[theorem]{Corollary}
\newtheorem{remark}[theorem]{Remark}

\numberwithin{equation}{section}

\begin{document}

\title {On backward stochastic differential equations approach to
valuation of American options}
\author {Tomasz Klimsiak and Andrzej Rozkosz}
\date{}
\maketitle

\renewcommand{\thefootnote}{}



\renewcommand{\thefootnote}{\arabic{footnote}}
\setcounter{footnote}{0}

\begin{abstract}
We consider the problem of valuation of American (call and put)
options written on a dividend paying stock governed by the
geometric Brownian motion. We show that the value function has two
different but related representations: by means of a solution of
some nonlinear backward stochastic differential equation and weak
solution to some semilinear partial differential equation.
\end{abstract}

\section{Introduction}

We consider a financial market model in which the price dynamics
of a dividend paying stock $X^{s,x}$ evolves (under the equivalent
martingale measure $P$) according to the stochastic differential
equation (SDE) of the form
\begin{equation}
\label{eq1.1} X^{s,x}_t=x+\int^t_s(r-d)X^{s,x}_{\theta}\,d\theta
+\int_s^t\sigma X^{s,x}_{\theta}\,dW_{\theta},\quad t\in[s,T].
\end{equation}
Here $x>0$, $W$ is a standard Wiener process, $d\ge0$ is the
dividend rate on the stock, $r\ge0$ is the risk-free interest rate
and $\sigma>0$ is the volatility.

It is well known (see, e.g., \cite[Section 2.5]{KS}) that the
arbitrage-free value of an American option with payoff function
$g:\BR\rightarrow[0,\infty)$ and expiration time $T$ is given by
\begin{equation}
\label{eq1.2} V(s,x)=\sup_{s\le\tau\le T}
Ee^{-r(\tau-s)}g(X^{s,x}_{\tau}),
\end{equation}
where $E$ denotes the expectation with respect to $P$ and the
supremum is taken over all stopping times with respect to the
standard augmentation $\{\FF_t\}$ of the filtration generated by
$W$. From \cite{EQ} we know also that the optimal stopping problem
and, a fortiori, the value function $V$, are related to the
solution $(Y^{s,x},Z^{s,x},K^{s,x})$ of the reflected backward
stochastic differential equation (RBSDE)
\begin{equation}
\label{eq1.3} \left\{
\begin{array}{l}
Y^{s,x}_t=g(X^{s,x}_T)-\int^T_trY^{s,x}_{\theta}\,d\theta
+K^{s,x}_T-K^{s,x}_t
-\int^T_tZ^{s,x}_{\theta}\,dW_{\theta},\quad
t\in[s,T],\medskip\\
Y^{s,x}_t\ge g(X^{s,x}_t),\quad t\in[s,T],\medskip \\
K^{s,x}\mbox{ is increasing, continuous, }K^{s,x}_s=0,\,\,
\int^T_s(Y^{s,x}_t-g(X^{s,x}_t))\,dK^{s,x}_t=0
\end{array}
\right.
\end{equation}
via the equality
\begin{equation}
\label{eq1.4} V(s,x)=Y^{s,x}_s,\quad (s,x)\in Q_T\equiv
[0,T]\times\BR.
\end{equation}
Formula (\ref{eq1.4}) when combined with general results on
connections between RBSDEs and parabolic PDEs proved in
\cite{EKPPQ} provides  a probabilistic proof of the fact that
$V=\{V(s,x); (s,x)\in Q_T\}$ with $V(s,x)$ given by (\ref{eq1.2})
is a viscosity solution of the obstacle problem (or, in another
terminology, the variational inequality)
\begin{equation}
\label{eq1.5}
\left\{
\begin{array}{ll}
\min(u(s,x)-g(x),-\LL_{BS}u(s,x)+ru(s,x))=0, &
(s,x)\in Q_T,\medskip\\
u(T,x)=g(x), & x\in\mathbb{R},
\end{array}
\right.
\end{equation}
where $\LL_{BS}$ is the Black and Scholes differential operator
defined by
\[
\LL_{BS}u=\partial_su+(r-d)x\partial_xu+\frac12\sigma^2x^2\partial^2_{xx}u.
\]

In the present paper we concentrate on the American call and put
options with exercise price $K>0$ for which the payoff function is
given by
\[
g(x)=\left\{
\begin{array}{ll} (x-K)^+, & \mbox{call option,} \\
(K-x)^+, & \mbox{put option. }
\end{array}
\right.
\]
We prove that in that case the process $K^{s,x}$ has the form
\begin{equation}
\label{eq1.6} K^{s,x}_t=\left\{
\begin{array}{ll} \int^t_s(dX^{s,x}_{\theta}-rK)^+
{\mathbf{1}}_{\{Y^{s,x}_{\theta}=g(X^{s,x}_{\theta})\}}\,d\theta,
 & \mbox{call option,} \medskip\\
\int^t_s(rK-dX^{s,x}_{\theta})^+
{\mathbf{1}}_{\{Y^{s,x}_{\theta}=g(X^{s,x}_{\theta})\}}\,d\theta,&
\mbox{put option}
\end{array}
\right.
\end{equation}
for $t\in[s,T]$, i.e. the first two components $(Y^{s,x},Z^{s,x})$
of the solution of (\ref{eq1.3}) solve the usual (non-reflected)
BSDE
\begin{align}
\label{eq1.7} Y^{s,x}_t&=g(X^{s,x}_T)+\int^T_t
(-rY^{s,x}_{\theta}+q(X^{s,x}_{\theta},Y^{s,x}_{\theta}))\,d\theta\\
&\quad-\int^T_tZ^{s,x}_{\theta}\,dW_{\theta},\quad
t\in[s,T],\nonumber
\end{align}
where
\[
q(x,y)=\left\{
\begin{array}{ll}
(dx-rK)^+{\mathbf{1}}_{(-\infty,g(x)]}(y), & \mbox{call option,}
\medskip\\
(rK-dx)^+{\mathbf{1}}_{(-\infty,g(x)]}(y),& \mbox{put option}
\end{array}
\right.
\]
for $x,y\in\mathbb{R}$. The above result is in fact a
reformulation of the representation for  Snell envelope of the
discounted payoff process $\xi_t=e^{-r(t-s)}g(X^{s,x}_t)$,
$t\in[s,T]$ (see Section \ref{sec3}). Therefore our contribution
here consists in providing new proof of the last statement and
clarifying relations between (\ref{eq1.3}) and (\ref{eq1.7}). We
also hope that our proof of the representation for Snell envelope
for $\xi$ will be of interest, because contrary to known to us
proofs it avoids considering the parabolic free-boundary value
problem associated with the optimal stopping problem
(\ref{eq1.2}).

Formula (\ref{eq1.6}) has an analytical counterpart. Let
$\varrho(x)=(1+|x|^2)^{-\alpha}$, $x\in\BR$, where $\alpha$ is
chosen so that $\int_{\BR}\varrho^2(x)x^2\,dx<\infty$. By a
solution of (\ref{eq1.5}) we understand a pair $(u,\mu)$
consisting of a measurable function $u:Q_T\rightarrow\BR$
possessing some regularity properties and a Radon measure $\mu$ on
$Q_{T}$ such that
\begin{equation}
\label{eq1.8} \left\{
\begin{array}{l}
\LL_{BS}u=ru-\mu,\medskip\\
u(T)=g,\quad u\ge g,\quad \int_{Q_T}(u-g)\varrho^2\,d\mu=0
\end{array}
\right.
\end{equation}
(see Section \ref{sec2} for details). We prove that (\ref{eq1.8})
has a unique solution $(u,\mu)$ such that $\mu$ is absolutely
continuous with respect to the Lebesgue measure and
\begin{equation}
\label{eq1.9} d\mu(t,x)=q(x,u(t,x))\,dt\,dx.
\end{equation}
Moreover, for each $(s,x)\in Q_{T}$ such that $x\neq0$,
\begin{equation}
\label{eq1.10} (Y^{s,x}_t,Z^{s,x}_t)=(u(t,X^{s,x}_t),\sigma x
\partial_xu(t,X^{s,x}_t)), \quad t\in[s,T],\quad P\mbox{-}a.s.,
\end{equation}
i.e. (\ref{eq1.3}) provides probabilistic representation for the
first component $u$ of a solution of (\ref{eq1.8}). In particular,
$V=u$. Formula (\ref{eq1.9}) is an analytical analogue of
(\ref{eq1.6}).

From (\ref{eq1.8}), (\ref{eq1.9}) it follows that $V$ is a
solution of the semilinear Cauchy problem
\begin{equation}
\label{eq1.11} \LL_{BS}u=ru-q(\cdot,u),\quad u(T,\cdot)=g.
\end{equation}
The above problem was considered in \cite{BKR1,BKR2} as an
alternative to the obstacle problem formulation (\ref{eq1.5}) and
the free boundary problem formulation (see, e.g., \cite[Section
2.7]{KS}). In \cite{BKR1} it is shown that (\ref{eq1.11}) has a
unique viscosity solution (since $q$ is discontinuous, the
standard definition of a viscosity solution is modified
appropriately) and $V=u$. Our approach to (\ref{eq1.5}) via
(\ref{eq1.8}) shows that in fact (\ref{eq1.11}) results from a
better understanding of the nature of solutions of (\ref{eq1.5}).

\section{Obstacle problem for the Black and Scholes equation}
\label{sec2}

In this section we prove existence, uniqueness and stochastic
representation of solutions of the obstacle problem (\ref{eq1.8}).
We begin with the precise definition of solutions of
(\ref{eq1.8}).

Let $Q_{st}=[s,t]\times\BR$, $Q_t=Q_{0t}$, and let $\mathcal{R}$
denote the space of all functions
$\varrho:\mathbb{R}\rightarrow\mathbb{R}$ of the form
$\varrho(x)=(1+|x|^{2})^{-\alpha}$, $x\in\mathbb{R}$, for some
$\alpha\ge0$. In the whole paper we assume that
$\int_{\mathbb{R}}\varrho^{2}(x)x^{2}\,dx<\infty$.

Given $\varrho\in\mathcal{R}$ we denote by
${\mathbb{L}}_{2,\varrho}({\mathbb{R}})$ the Hilbert  space of
functions $u$ on ${\mathbb{R}}$ such that $u\varrho\in
{\mathbb{L}}_2({\mathbb{R}})$ equipped with the inner product
$\langle u,v\rangle_{2,\varrho}
=\int_{\mathbb{R}}uv\varrho^{2}\,dx$. Similarly, by
${\mathbb{L}}_{2,\varrho}(Q_{st})$ we denote the Hilbert space of
functions $u$ on $Q_{st}$ such that
$u\varrho\in{\mathbb{L}}_{2}(Q_{st})$ with the inner product
$\langle u,v\rangle_{2,\varrho,s,t}
=\int_{Q_{st}}uv\varrho^{2}\,dx\,dt$. If $s=0$ we drop the
subscript $s$ in the notation. $H_{\varrho}=\{\eta\in
{\mathbb{L}}_{2,\varrho}({\mathbb{R}}): x\partial_x\eta(x)\in{
\mathbb{L}}_{2,\varrho}({\mathbb{R}})\}$, $W_{\varrho}=\{\eta\in
{\mathbb{L}}_{2}(0,T;H_{\varrho}):
\partial_t\eta\in{\mathbb{L}}_{2}(0,T;H_{\varrho}^{-1})\}$, where
$H^{-1}_{\varrho}$ is the space dual to $H_{\varrho}$. By $\langle
\cdot,\cdot\rangle_{\varrho,T}$ we denote the duality pairing
between ${\mathbb{L}}_{2}(0,T;H_{\varrho})$ and
${\mathbb{L}}_{2}(0,T;H_{\varrho}^{-1})$. Finally,
$V=W_{\varrho}\cap C(Q_{T})$.

We say that a pair $(u,\mu)$, where $u\in V$ and $\mu$ is a Radon
measure on $Q_{T}$, is a solution of the obstacle problem
(\ref{eq1.8}) if
\begin{equation}
\label{eq2.01} u(T)=g,\quad u\ge g,\quad
\int_{Q_T}(u-g)\varrho^2\,d\mu=0
\end{equation}
and the equation
\begin{equation}
\label{eq2.02} \LL_{BS}u=ru-\mu
\end{equation}
is satisfied in the strong sense, i.e. for every $\eta\in
C_{0}^{\infty}(Q_{T})$,
\[
\langle\partial_tu,\eta\rangle_{\varrho,T}+ \langle
\LL_{BS}u,\eta\rangle_{\varrho,T} =r\langle
u,\eta\rangle_{2,\varrho,T}-\int_{Q_{T}}\eta\varrho^{2}\,d\mu,
\]
where
\[
\langle\LL_{BS}u,\eta\rangle_{\varrho,T}
=\langle(r-d)x\partial_xu,\eta\rangle_{2,\varrho,T}
-\frac12\sigma^{2}\langle\partial_xu,
\partial_x(x^{2}\eta\varrho^{2})\rangle_{2,T}.
\]

We  say that a pair $(u,\mu)$ satisfies (\ref{eq2.02}) in the weak
sense if $\mu$ is a Radon measure on $Q_T$, $u\in\BL_{2}(0,T;
H_{\varrho})\cap C([0,T],{\mathbb{L}}_{2,\varrho}(\mathbb{R}))$
and for every $\eta\in C_{0}^{\infty}(Q_{T})$,
\begin{align*}
\langle u,\partial_t\eta\rangle_{\varrho,T} -\langle
\LL_{BS}u,\eta\rangle_{\varrho,T}&=\langle h(T),
\eta(T)\rangle_{2,\varrho}-\langle u(0),
\eta(0)\rangle_{2,\varrho}\\
&\quad-r\langle u,\eta\rangle_{2,\varrho,T}
+\int_{Q_{T}}\eta\varrho^{2}\,d\mu.
\end{align*}

Let $\{\FF_t\}$ denote the standard augmentation of the natural
filtration generated by $W$. By a solution of RBSDE (\ref{eq1.3})
we understand a triple $(Y^{s,x},Z^{s,x},K^{s,x})$ of
$\{\FF_t\}$-progressively measurable processes on $[s,T]$ such
that
\begin{equation}
\label{eq2.1} E\sup_{t\in[s,T]}|Y^{s,x}_t|^2<\infty,\quad
E\int^T_s|Z^{s,x}_t|^2\,dt<\infty,\quad E|K^{s,x}_T|^2<\infty
\end{equation}
and (\ref{eq1.3}) is satisfied $P$-a.s.. A pair
$(Y^{s,x},Z^{s,x})$ of $\{\FF_t\}$-progressively measurable
process is a solution of BSDE (\ref{eq1.7}) if (\ref{eq1.7}) holds
$P$-a.s. and $Y^{s,x},Z^{s,x}$ satisfy the integrability
conditions (\ref{eq2.1}).

From general results proved in \cite{EKPPQ} it follows that
(\ref{eq1.3}) has a unique solution. We shall prove that the third
component $K^{s,x}$ of the solution is absolutely continuous.

\begin{proposition}
\label{prop2.1}  If $(Y^{s,x},Z^{s,x},K^{s,x})$ is a solution of
RBSDE (\ref{eq1.3}) then
\begin{equation}
\label{eq3.3} K^{s,x}_{t}-K^{s,x}_{\tau}\le
\int_{\tau}^{t}{\mathbf{1}}_{\{Y^{s,x}_\theta
=S_\theta\}}(dX^{s,x}_\theta-rK)^+\,d\theta,\quad s\le\tau\le t\le
T.
\end{equation}
\end{proposition}
\begin{proof}
We prove the theorem in the case of call option. The proof for put
option is similar and therefore left to the reader.

Suppose that $(Y^{s,x},Z^{s,x},K^{s,x})$ is a solution of
(\ref{eq1.3}) and $u$ is a viscosity solution of (\ref{eq1.5}). By
\cite[Theorem 8.5]{EKPPQ},
\begin{equation}
\label{eq3.01} Y^{s,x}_t=u(t,X^{s,x}_t),\quad t\in[s,T].
\end{equation}
Set $S_t=g(X^{s,x}_t)$, $t\in[s,T]$, and denote by
$\{L^0_t(\xi);t\ge0\}$ the local time at 0 of a continuous
semimartingale $\xi$.  By the Tanaka-Meyer formula, for
$t\in[s,T]$ we have
\begin{align}
\label{eq3.02} (X^{s,x}_t-K)^+
&=\int_s^t{\mathbf{1}}_{(K,\infty)}(X^{s,x}_\theta)(r-d)X^{s,x}_\theta\,
d\theta \\
&\quad+\int_s^t{\mathbf{1}}_{(K,\infty)}(X^{s,x}_\theta)\sigma
X^{s,x}_\theta \,dW_{\theta}+\frac12L^{0}_t(X^{s,x}-K) \nonumber
\end{align}
and
\begin{align}
\label{eq2.3} 0&=(Y^{s,x}_t-S_t)^-
=-\int_s^t{\mathbf{1}}_{(-\infty,0]}
(Y^{s,x}_\theta-S_{\theta})\,dY^{s,x}_\theta \\
&\quad +\int_s^t{\mathbf{1}}_{(-\infty,0]}
(Y^{s,x}_\theta-S_{\theta})\,dS_\theta
+\frac12L^{0}_t(Y^{s,x}-S)\nonumber\\
&=\int_s^t{\mathbf{1}}_{\{Y^{s,x}_{\theta}=S_{\theta}\}}(-rY^{s,x}_\theta
\,d\theta+dK^{s,x}_{\theta}-Z^{s,x}_\theta \,dW_{\theta})\nonumber\\
&\quad+\int_s^t{\mathbf{1}}_{(K,\infty)}(X^{s,x}_\theta)
{\mathbf{1}}_{\{Y^{s,x}_{\theta}=S_{\theta}\}}((r-d)X^{s,x}_\theta
\,d\theta+\sigma X^{s,x}_\theta\,dW_{\theta})\nonumber\\
&\quad+\frac12\int_s^t{\mathbf{1}}_{\{Y^{s,x}_{\theta}=S_{\theta}\}}\,
dL_\theta^{0}(X^{s,x}-K) +\frac12L^{0}_t(Y^{s,x}-S). \nonumber
\end{align}
Write $I=\{u=g\}$ and observe that $(t,K)\notin I$ for all
$t\in[0,T)$, because $u=V$ by \cite[Proposition 2.3]{EKPPQ} and
hence $u$ is strictly positive. Consequently,
\[
\int_s^t{\mathbf{1}}_{\{Y^{s,x}_{\theta}=S_{\theta}\}}
\,dL_\theta^{0}(X^{s,x}-K)=0.
\]
Furthermore, from (\ref{eq3.02}) and Proposition 4.2 and Remark
4.3 in \cite{EKPPQ} it follows that $\sigma
X^{s,x}_t{\mathbf{1}}_{(K,\infty)}(X^{s,x}_t)=Z^{s,x}_t$ $P$-a.s.
on $\{Y^{s,x}_t=S_t\}$. From (\ref{eq2.3}) we therefore get
\begin{align*}
&K^{s,x}_t-K^{s,x}_{\tau}+\frac12L^{0}_t(Y^{s,x}-S)
-\frac12L^{0}_{\tau}(Y^{s,x}-S)\\
&\qquad=\int_{\tau}^tr{\mathbf{1}}_{\{Y^{s,x}_\theta
=S_\theta\}}S_\theta\,d\theta
-\int_{\tau}^t{\mathbf{1}}_{\{Y^{s,x}_\theta
=S_\theta\}}{\mathbf{1}}_{(K,\infty)}
(X^{s,x}_{\theta})(r-d)X^{s,x}_\theta\,d\theta\\
&\qquad=\int_{\tau}^t{\mathbf{1}}_{\{Y^{s,x}_\theta
=S_\theta\}}{\mathbf{1}}_{(K,\infty)}(X^{s,x}_{\theta})
((r-d)X^{s,x}_\theta-r(X^{s,x}_\theta-K)^+)^-\,d\theta.
\end{align*}
Hence
\begin{equation}
\label{eq2.6} K^{s,x}_t-K^{s,x}_{\tau}\le
\int_{\tau}^t{\mathbf{1}}_{\{Y^{s,x}_\theta
=S_\theta\}}{\mathbf{1}}_{(K,\infty)}(X^{s,x}_\theta)
((r-d)X^{s,x}_\theta-r(X^{s,x}_\theta-K)^+)^-\,d\theta.
\end{equation}
Since, by (\ref{eq3.01}), $Y^{s,x}$ is strictly positive,
$\{Y_{t}^{s,x}=g(X^{s,x}_{t})\}\subset \{X^{s,x}_{t}>K\}$ and
hence $K^{s,x}$ increases only on the set $\{X^{s,x}_t>K\}$.
Therefore (\ref{eq2.6}) forces (\ref{eq3.3}).
\end{proof}

\begin{proposition}
\label{prop2.2} There exists at most one solution of the problem
(\ref{eq1.8}).
\end{proposition}
\begin{proof}
Suppose that $(u_{1},\mu_{1})$, $(u_{2},\mu_{2})$ are two
solutions of (\ref{eq1.8}). Write $u=u_{1}-u_{2}$,
$\mu=\mu_{1}-\mu_{2}$. Then $(u,\mu)$ satisfies (\ref{eq2.02}) in
the strong sense. Since by standard regularization arguments we
can put $u$ as a test function in (\ref{eq2.02}) and obviously
(\ref{eq2.02}) is satisfied on $Q_{tT}$ for any $t\in[0,T)$, we
have
\begin{align*}
&\|u(t)\|^2_{2,\varrho}+\frac12\sigma^{2}\|x\partial_xu\|^{2}_{2,\varrho,t,T}
=\langle (\mu-d)x\partial_xu,u\rangle_{2,\varrho,t,T}
+\sigma^{2}\langle\partial_xu,xu\rangle_{2,\varrho,t,T}\\
&\qquad
+\sigma^{2}\langle\partial_xu,x^{2}u\partial_x\varrho,\varrho\rangle_{2,t,T}
+r\|u\|^{2}_{2,\varrho,t,T}+\int_{Q_{tT}}u\,d\mu.
\end{align*}
From the above, the fact that $\int_{Q_{tT}}u\,d\mu\le 0$,
$|\partial_x\varrho|\le C\varrho$ and the elementary inequality
$ab\le\varepsilon a^{2}+\varepsilon^{-1}b^{2}$ we get
\[
\|u(t)\|^2_{2,\varrho}\le
C\int_{t}^{T}\|u(s)\|^{2}_{2,\varrho}\,ds,\quad t\in[0,T].
\]
By Gronwall's lemma, $u=0$, and in consequence, $\mu=0$.
\end{proof}
\medskip

Given $\delta>0$ write $D^{+}_{\delta}=(0,T)\times (\delta,
+\infty)$, $D^{-}_{\delta}=(0,T)\times (-\infty,\delta)$ and
$D^{+}=D^{+}_{0}$, $D^{-}=D^{-}_{0}$, $D=D^{+}\cup D^{-}$. Note
that from the well known explicit formula for $X^{s,x}$ it follows
that $X^{s,x}_{t}\in D^{+}$, $t\in[s,T]$, $P$-a.s. if $x>0$, and
$X^{s,x}_{t}\in D^{-}$, $t\in[s,T]$, $P$-a.s. if $x<0$. Note also
that if $x\neq 0$ and  $t>s$  then the  distribution density of
the random variable $X_t^{s,x}$ is given by the formula
\begin{equation}
\label{eqa.0}
p(s,x,t,y)=\frac{1}{y\sqrt{2\pi(t-s)}}\exp(\frac{-(\ln\frac{y}{x}
+(\frac{\sigma^{2}}{2}-r+d)(t-s))^{2}}{t-s})
{\mathbf{1}}_{\{\frac{y}{x}>0\}}.
\end{equation}
It follows in particular that for fixed $s\in[0,T)$, $x\neq 0$ and
$\delta\in (0,T-s]$ the function $p(s,x,\cdot,\cdot)$ is bounded
on $Q_{s+\delta,T}$.

\begin{theorem}
\label{th2.3}
$(i)$ There exists a unique solution $(u,\mu)$ of the problem
(\ref{eq1.8}).\\
$(ii)$ Let $x\neq0$ and let $(Y^{s,x},Z^{s,x},K^{s,x})$ be a
solution of RBSDE (\ref{eq1.3}). Then
\[
(Y^{s,x}_t,Z^{s,x}_t)=(u(t,X_{t}^{s,x}),\sigma\partial_x
u(t,X_{t}^{s,x})),\quad t\in[s,T],\quad P\mbox{-}a.s.
\]
and  for any $\eta\in C_{0}(Q_{sT})$,
\begin{equation}
\label{eq2.7} E\int_{s}^{T}\eta(t,X_t)\,dK^{s,x}_t
=\int_{Q_{sT}}\eta(t,y) p(s,x,t,y)\,d\mu(t,y).
\end{equation}
\end{theorem}
\begin{proof}
By \cite[Theorem 2.2]{Pa} for each $n\in\BN$ there exists a unique
viscosity solution $u_n$ of the following penalized problem
\begin{equation}
\label{eqa.9} \frac{\partial u_{n}}{\partial
t}+\LL_{BS}u_{n}=ru_{n}-n(u_{n}-g)^{-},\quad u_{n}(T)=g.
\end{equation}
Let $(Y^{s,x,n},Z^{s,x,n})$ denote a solution of the BSDE
\begin{align*}
Y^{s,x,n}_{t}&=g(X^{s,x}_{T})-\int_{t}^{T}rY^{s,x,n}_{\theta}\,d\theta
+\int_{t}^{T}n(Y^{s,x,n}_{\theta}-g(X^{s,x}_{\theta}))^{-}\,d\theta\\
&\quad-\int_{t}^{T}Z^{s,x,n}_{\theta}\,dW_{\theta}.
\end{align*}
Using standard arguments one can show that $x\mapsto EY^{s,x,n}_s$
is Lipschitz continuous uniformly in $s$. Therefore $u_n$ has the
same regularity, because by \cite[Theorem 2.2]{Pa},
$Y^{s,x,n}_{t}=u_{n}(t,X^{s,x}_{t})$, $t\in [s,T]$, $P$-a.s., and
hence  $u_n(s,x)=EY^{s,x,n}_s$.  Since the operator $L_{BS}$ is
uniformly elliptic on each domain $D^{+}_{\delta}$, for each
$\delta>0$ there is a unique weak solution $v_{\delta}$ of the
following terminal-boundary value problem
\[
\frac{\partial v_{\delta}}{\partial t}+\LL_{BS}v_{\delta}
=rv_{\delta}-n(v_{\delta}-g)^{-},\,\,v_{\delta}(T)=g, \,\,
v_{\delta}(t,x)=u_n(t,x) \mbox{ on }[0,T]\times\{\delta\}
\]
(see \cite[Theorem V.6.1]{Lad}). Since $v_{\delta}$ is a viscosity
solution of the above problem as well,
$v_{\delta}=u_{n|D^{+}_{\delta}}$ by uniqueness. Using this,
Lipschitz continuity of $u_{n}$ and \cite[Theorem 1.5.9]{Friedman}
we conclude that $u_{n}\in C^{1,2}(D)$. Hence, by Proposition
1.2.3 and Theorem 2.2.1 in \cite{Nualart},
$Y^{s,x,n}_{t}\in{\mathbb{D}}^{1,2}$ for every $(s,x)\in Q_{T}$
such that $x\neq0$,  where ${\mathbb{D}}^{1,2}$ is the domain of
the derivative operator in ${\mathbb{L}}_2(\Omega)$ (see
\cite[Section 1.2]{Nualart} for a precise definition).
Consequently, applying once again Proposition 1.2.3 and Theorem
2.2.1 in \cite{Nualart} and using the fact that $g$ and $x\mapsto
x^{-}$ are Lipschitz continuous functions we conclude that if
$x\neq0$ then $g(X^{s,x}_{T})$,
$\int_{t}^{T}rY^{s,x,n}_{\theta}\,d\theta$,
$\int_{t}^{T}n(Y^{s,x,n}_{\theta}
-g(X^{s,x}_{\theta}))^{-}\,d\theta\in{\mathbb{D}}^{1,2}$.
Moreover, by \cite[Proposition 1.2.3]{Nualart} and \cite[Lemma
5.1]{EPQ}, there exists an adapted bounded process $A$ such that
for every $s<\tau\le t$,
\begin{align*}
D_{\tau}Y_{t}^{s,x,n}&=Z^{s,x,n}_{\tau}
+\int_{\tau}^{t}D_{\tau}Z^{s,x,n}_{\theta}\,d\theta
+r\int_{\tau}^{t}D_{\tau}Y^{s,x,n}_{\theta}\,d\theta\\
&\quad-n\int_{\tau}^{t}A_{\theta}
D_{\tau}(Y^{s,x,n}_{\theta}-g(X_{\theta}))\,d\theta,
\end{align*}
where $D_{\tau}$ denotes the derivative operator. From this it
follows in particular that
\[
D_{t}Y^{s,x,n}_{t}=Z^{s,x,n}_{t},\quad P\mbox{-}a.s.
\]
for every $t\in[s,T]$. On the other hand, by remarks following the
proof of Proposition \ref{prop2.2} and remark following the proof
of \cite[Proposition 1.2.3]{Nualart},
\[
D_{\tau}Y_{t}^{s,x,n} =\partial_x
u_{n}(t,X^{s,x}_{t})D_{\tau}X^{s,x}_{t},\quad P\mbox{-}a.s.
\]
for every $r,t\in [s,T]$. Moreover, by \cite[Theorem
2.2.1]{Nualart}, $D_{t}X^{s,x}_{t}=\sigma X_{t}^{s,x}$. Thus, if
$x\neq0$, then
\[
Z^{s,x,n}_{t}=\sigma X^{s,x}_{t}\partial_x u_{n}(t,X_{t}),\quad
P\mbox{-}a.s..
\]
By results from Section 6 in \cite{EKPPQ} and standard estimates
for diffusions we have
\begin{align}
\label{eqa.10} &E\sup_{s\le t\le T}|u_{n}(t,X^{s,x}_{t})|^{2}
+E\int_{s}^{T}|\sigma X^{s,x}_t\partial_x
u_{n}(t,X^{s,x}_t)|^{2}\,dt\\
&\qquad\le CE\sup_{s\le t\le T}|g(X^{s,x}_{t})|^{2}\le
C|x|^{2}.\nonumber
\end{align}
By the above and \cite[Proposition 5.1]{BM} it follows that
$u_{n}\in \BL_{2}(0,T;H_{\varrho})$. Accordingly, $u_{n}$ is a
weak solution of (\ref{eqa.9}). Furthermore, from results proved
in \cite[Section 6]{EKPPQ} it follows that for every $(s,x)\in
Q_{T}$,
\begin{align}
\label{eqa.11} &E\sup_{s\le t\le T}
|(u_{n}-u_{m})(t,X^{s,x}_{t})|^2 +E\int_{s}^{T}|\sigma
X^{s,x}_t\partial_x (u_{n}-u_{m})(t,X^{s,x}_t)|^2\,dt \\
&\qquad+ E\sup_{s\le t\le T}
|K^{s,x,n}_{t}-K^{s,x,m}_{t}|^{2}\rightarrow 0\nonumber
\end{align}
as $m,n\rightarrow\infty$. From (\ref{eqa.10}), (\ref{eqa.11}) and
\cite[Proposition 5.1]{BM}  we conclude that there exists $u\in
C(Q_{T})\cap \BL_{2}(0,T;H_{\varrho})$ such that $u_{n}\rightarrow
u$ uniformly on compact subsets of $Q_{T}$, $u_{n}\rightarrow u$
in $\BL_{2}(0,T;H_{\varrho})$ and $u_{n}\rightarrow u$ in
$C([0,T],{\mathbb{L}}_{2,\varrho}(\mathbb{R}))$. Moreover, using
(\ref{eqa.10}) and  once again \cite[Proposition 5.1]{BM} we see
that $\|u_{n}\|_{\BL_{2}(0,T;H_{\varrho})}\le C$. Therefore from
(\ref{eqa.9}) it follows that the sequence of measures $\{\mu_n\}$
defined by $d\mu_{n}=n(u_{n}-g)^{-}\,d\lambda$, $n\in\BN$,  where
$\lambda$ is the 2-dimensional Lebesgue measure, is tight. If
$\mu_{n}\rightarrow\mu$ weakly, which we may assume, then letting
$n\rightarrow\infty$ in (\ref{eqa.9}) we conclude that the pair
$(u,\mu)$ satisfies equation (\ref{eq2.02}) in the weak sense and
that
\[
u(t,X_{t}^{s,x})=Y^{s,x}_{t},\, t\in [s,T],\, P\mbox{-}a.s.,\quad
Z^{s,x}_{t} =\sigma X_{t}^{s,x}\partial_x u(t,X_{t}^{s,x}),\,
dt\otimes P\mbox{-}a.s.
\]
because in \cite[Section 6]{EKPPQ} it is proved that
$Y^{s,x,n}_t\rightarrow Y^{s,x}_t$, $t\in[s,T]$, $P$-a.s. and
$E\int^T_s|Z^{s,x,n}_t-Z^{s,x}_t|^2\,dt\rightarrow0$. In
particular, it follows from the above that $u\ge g$. Let $\eta\in
C_{0}(Q_{T})$. Since $u_{n}\rightarrow u$ uniformly,
\[
\int_{Q_{T}}(u_{n}-g)\eta\,d\mu_{n}\rightarrow
\int_{Q_{T}}(u-g)\eta\,d\mu\ge 0.
\]
On the other hand,
\[
\int_{Q_{T}}(u_{n}-g)\eta\,d\mu_{n}
=-\int_{Q_{T}}n((u_{n}-g)^{-})^{2}\,d\lambda\le 0.
\]
From this we get (\ref{eq2.01}). Furthermore, if $x\neq0$ then for
any $\delta\in (0,T-s)$ and $\eta\in C_{0}(Q_{s+\delta,T})$ we
have
\begin{equation}
\label{eq2.12} E\int_{s}^{T}\eta(t,X_t^{s,x})\,dK^{s,x,n}_t
=\int_{Q_{sT}}\eta(t,y)p(s,x,t,y)\,d\mu_{n}(t,y).
\end{equation}
Since it is known that $K^{s,x,n}_t\rightarrow K^{s,x}_t$
uniformly in $t\in[s,T]$ in probability (see \cite[Section
6]{EKPPQ}), letting $n\rightarrow\infty$ in (\ref{eq2.12}) and
using (\ref{eqa.0}), (\ref{eqa.11}) we get (\ref{eq2.7}) for
$\eta\in C_{0}(Q_{s+\delta,T})$, and hence for any $\eta\in
C_{0}(Q_{sT})$. In order to complete the proof we have to show
that $u\in W_{\varrho}$. Since $p(s,x,\cdot,\cdot)$ is positive
for every $(s,x)\in Q_{T}$ such that $x\neq0$, it follows from
(\ref{eq2.7}) and Proposition \ref{prop2.1} that
$d\mu\le\fch_{\{u=g\}}(t,x) (dx-rK)^{+}\,d\lambda$, i.e. for every
$\eta\in C_{0}^{+}(Q_{T})$,
\[
\int_{Q_{T}}\eta(t,x)\,d\mu(t,x)\le
\int_{Q_{T}}\eta(t,x)\fch_{\{u=g\}}(t,x)(dx-rK)^{+}\, dx\,dt.
\]
Hence there exists a measurable function $\alpha$ on $Q_T$ such
that $0\le\alpha\le 1$ and
\begin{eqnarray}
\label{eqa.13}
\frac{d\mu}{d\lambda}(t,x)=\alpha(t,x)\fch_{\{u=g\}}(t,x)(dx-rK)^{+}.
\end{eqnarray}
This implies that $u\in W_{\varrho}$ and $u$ satisfies
(\ref{eq2.02}) in the strong sense, i.e. $(u,\mu)$ is a solution
of (\ref{eq1.8}).
\end{proof}

\begin{remark}
\label{rema.1} {\rm It is known  that $\{u=g\}=\{(t,x)\in Q_T:
x\ge s(t)\}$ for some nonincreasing function $s$ in the case of
call option and $\{u=g\}=\{(t,x)\in Q_T: 0\le x\le s(t)\}$ for
some nondecreasing $s$ in the case of put option (see, e.g.,
\cite[Proposition 2.7.6]{KS}). It follows that in both cases the
2-dimensional Lebesgue measure of the boundary of $\{u=g\}$ equals
zero. }
\end{remark}

\section{Linear RBSDEs and nonlinear BSDEs} \label{sec3}

We begin with proving the key formulas (\ref{eq1.6}),
(\ref{eq1.9}). As first application we will show the
semimartingale representation for the Snell envelope of the
discounted payoff process and the early exercise premium
representation for $V$.
\begin{theorem}
\label{th3.1} $(i)$  If $(u,\mu)$ is a solution of the obstacle
problem (\ref{eq1.8}), then $\mu$ is given by (\ref{eq1.9}). \\
$(ii)$ If $(Y^{s,x},Z^{s,x},K^{s,x})$ is a solution of 
(\ref{eq1.3}), then $K^{s,x}$ is given by (\ref{eq1.6}).
\end{theorem}
\begin{proof}
We prove the theorem in the case of call option. The proof for put
option requires only some obvious changes and is left to the
reader.

Suppose that $(Y^{s,x},Z^{s,x},K^{s,x})$ is a solution of
(\ref{eq1.3}) and $(u,\mu)$ is a solution of (\ref{eq1.8}). By
(\ref{eqa.13}), $u$ solves  the equation
\begin{equation}
\label{eq3.6}
\partial_tu+(r-d)x\partial_xu+\frac12\sigma^2x^2\partial^2_{xx}u
=ru-\alpha(t,x){\mathbf{1}}_{\{u=g\}}(t,x)(dx-rK)^+
\end{equation}
in the strong sense. Let $I=\{u=g\}$ and $I_0=$Int\,$I$. If
$I_0\neq\emptyset$ then by (\ref{eq3.6}), for any $\eta\in
C_0^\infty(I_0)$ we have
\begin{align*}
&\int_{Q_T}u(t,x)\partial_t\eta(t,x)\,dt\,dx -\frac12\int_{Q_T}
\sigma^2x^2\partial^2_{xx}u
(t,x)\eta(t,x)\,dt\,dx\\
&\qquad\qquad\qquad-\int_{Q_T}(r-d)x\partial_xu(t,x)\eta(t,x)\,dt\,dx\\
&\qquad=\int_{Q_T}(-ru(t,x)+\alpha(t,x){\mathbf{1}}_{\{u=g\}}(t,x)
(dx-rK)^+)\eta(t,x)\,dt\,dx\\
&\qquad\quad+\int_{\mathbb{R}}g(x)\eta(T,x)\,dx
-\int_{\mathbb{R}}u(0,x)\eta(0,x)\,dx.
\end{align*}
Since supp\,$\eta\subset I_0$ and $g$ is regular on $I_0$, we
deduce from the above that
\begin{align*}
&\int_{I_0}(r-d)x{\mathbf{1}}_{[K,\infty)}(x)\eta(t,x)\,dt\,dx
=\int_{I_0}rg(x)\eta(t,x)\,dt\,dx\\
&\qquad-\int_{I_0}\alpha(t,x)
{\mathbf{1}}_{\{u=g\}}(t,x)(dx-rK)^+\eta(t,x)\,dt\,dx.
\end{align*}
Equivalently, we have
\begin{align*}
&\int_{I_0}f(t,x)\eta(t,x)\,dt\,dx\\
&\quad=\int_{I_0}\alpha(t,x){\mathbf{1}}_{\{u=g\}}(t,x)
{\mathbf{1}}_{[K,\infty)}(x)(dx-rK)^+\eta(t,x)\,dt\,dx,
\end{align*}
where $f(t,x)=(r-d)x{\mathbf{1}}_{[K,\infty)}(x)-r(x-K)^+
=(-dx+rK){\mathbf{1}}_{[K,\infty)}(x)$. Since
\begin{align*}
\alpha(t,x)(dx-rK)^+&=-\alpha(t,x)((r-d)x{\mathbf{1}}_{[K,\infty)}(x)
-r(x-K)^+)^-\\
&=-\alpha(t,x)f^-(x)
\end{align*}
on $I_0$, it follows that
\[
\int_{I_0}f(t,x)\eta(t,x)\,dt\,dx
=-\int_{I_0}\alpha(t,x)f^-(t,x)\eta(t,x)\,dt\,dx
\]
for any $\eta\in C_0^\infty(I_0)$. Hence
$f(t,x)=-\alpha(t,x)f^-(t,x)$ a.e. on $I_0$. Since $f=f^+-f^-$, we
have $f^+(t,x)=(1-\alpha(t,x))f^-(t,x)$, and consequently,
$(1-\alpha(t,x))f^-(t,x)=0$ a.e. on $I_0$, i.e.
$\alpha(t,x)(dx-rK)^+=(dx-rK)^+$ a.e. on $I_0$. Since by Remark
\ref{rema.1} the Lebesgue measure of $\partial I$ equals zero, the
above equality holds a.e. on $I$, which in view of (\ref{eqa.13})
completes the proof of (i).

In case $x=0$ part (ii) is trivial since in that case
$X^{s,x}_t=K^{s,x}_t=0$, $t\in[s,T]$. In case $x\neq0$ part (ii)
follows from part (i) and results proved in \cite{K}. To see this,
let us denote by $X$ the canonical process on the space
$C([0,T];\BR)$ of continuous functions on $[0,T]$, and by
$P_{s,x}$ the law of $X^{s,x}$, i.e.
$P_{s,x}=P\circ(X^{s,x})^{-1}$. We may and will assume that
$X^{s,x}_s=x$, $t\in[0,s]$, and hence that $P_{s,x}$ is a measure
on $C([0,T];\BR)$. Write
\[
M_{s,t}=X_t-X_s-\int^t_s(r-d)X_{\theta}\,d\theta,\quad
B_{s,t}=\int^t_s\frac1{\sigma X_{\theta}}\,dM_{s,\theta}, \quad
0\le s\le t\le T
\]
and observe that if $x\neq0$ then under $P_{s,x}$ the process
$B_{s,\cdot}$ is a standard Wiener process on $[s,T]$ with respect
to the natural filtration generated by $X$. Furthermore, for $0\le
s<t\le T$ set
\[
K_{s,t}=u(s,X_s)-u(t,X_t) +\int^t_sru(\theta,X_{\theta})\,d\theta
+\int^t_s\sigma\partial_x u(\theta,X_{\theta})\,dB_{s,\theta}
\]
and
\[
\tilde K_{s,t}=\int^t_s(dX_{\theta}-rK)^+
\fch_{\{u(\theta,X_{\theta})=g(X_{\theta})\}} \,d\theta.
\]
Let $(Y^{s,x},Z^{s,x},K^{s,x})$ be  a solution of (\ref{eq1.3})
and let $\tilde K^{s,x}$ denote the process defined by the right
hand-side of (\ref{eq1.6}). By Theorem \ref{th2.3}, for every
$(s,x)\in[0,T)\times\BR$,
\begin{align*}
K^{s,x}_t-K^{s,x}_s&=u(s,X^{s,x}_s)-u(t,X^{s,x}_t)
+\int^t_sru(\theta,X^{s,x}_{\theta})\,d\theta\\
&\quad+\int^t_s\sigma\partial_x
u(\theta,X^{s,x}_{\theta})\,dW_{\theta},\quad 0\le s<t\le T,\quad
P\mbox{-}a.s..
\end{align*}
From this and the fact that the law of $(X,B_{s,\cdot})$ under
$P_{s,x}$ is equal to the law of $(X^{s,x},W_{\cdot}-W_s)$ under
$P$ we conclude that the law of $K_{s,\cdot}$ under $P_{s,x}$ is
equal to the law of $K^{s,x}$ under $P$. Consequently, by
(\ref{eq2.7}), for every $s\in[0,T)$, $x\neq0$,
\begin{equation}
\label{eq3.2} E_{s,x}\int_{s}^{T}\eta(t,X_t)\,dK_{s,t}
=\int_{Q_{sT}}\eta(t,y) p(s,x,t,y)\,d\mu(t,y)
\end{equation}
for all $\eta\in C_{0}(Q_{sT})$, where $E_{s,x}$ denotes the
expectation with respect to $P_{s,x}$. Thus, the additive
functional $K=\{K_{s,t};0\le s\le t\le T\}$ of the Markov family
$\{(X,P_{s,x});(s,x)\in[0,T)\times\BR\}$ corresponds to the
measure $\mu$ in the sense defined in \cite{K}. Similarly, for
every $s\in[0,T)$, $x\neq0$ the law of $\tilde K_{s,\cdot}$ under
$P_{s,x}$ is equal to the law of $\tilde K^{s,x}$ under $P$, and
hence, by part (i), (\ref{eq3.2}) is satisfied with $K$ replaced
by $\tilde K$, i.e. the additive functional $\tilde K=\{\tilde
K_{s,t};0\le s\le t\le T\}$ corresponds to $\mu$, too. By
\cite[Corollary 6.6]{K}, $P_{s,x}(K_{s,t}=\tilde
K_{s,t},t\in[s,T])=1$ for every $s\in[0,T)$, $x\neq0$. Hence
$P(K^{s,x}_t=\tilde K^{s,x}_t,t\in[s,T])=1$ for $s\in[0,T)$,
$x\neq0$, which completes the proof.
\end{proof}

\begin{corollary}
If $(Y^{s,x},Z^{s,x},K^{s,x})$ is a solution of (\ref{eq1.3}) then
$(Y^{s,x},Z^{s,x})$ is a solution of (\ref{eq1.7}). Conversely, if
$(Y^{s,x},Z^{s,x})$ is a solution of (\ref{eq1.7}) then
$(Y^{s,x},Z^{s,x},K^{s,x})$ with $K^{s,x}$ defined by
(\ref{eq1.6}) is a solution of (\ref{eq1.3}).
\end{corollary}
\begin{proof}
The first part follows immediately from Theorem \ref{th3.1}. The
second part is a consequence of the first one and the fact that
the solution of (\ref{eq1.7}) is unique, because for every
$x\in\mathbb{R}$ the function $y\mapsto q(x,y)$ is decreasing.
\end{proof}
\medskip

Let $\xi$ denote the discounted payoff process for the American
option, i.e.
\[
\xi_t=e^{-r(t-s)}g(X^{s,x}_t),\quad t\in[s,T].
\]
By (\ref{eq1.7}),
\begin{align*}
e^{r(t-s)}Y^{s,x}_t&=e^{-r(T-s)}g(X^{s,x}_T) +\int^T_t
e^{-r(\theta-s)}q(X^{s,x}_{\theta},Y^{s,x}_{\theta})\,d\theta\\
&\quad-\int^T_te^{-r(\theta-s)}Z^{s,x}_{\theta}\,dW_{\theta}.
\end{align*}
From this and the fact that
$V(t,X^{s,x}_t)=u(t,X^{s,x}_t)=Y^{s,x}_t$, $t\in[s,T]$, we obtain
\begin{corollary}
The Snell envelope $\eta_t=e^{-r(t-s)}V(t,X^{s,x}_t)$,
$t\in[s,T]$, of $\xi$ admits the representation
\begin{equation}
\label{eq3.7} \eta_t=E\left(e^{-r(T-s)}g(X^{s,x}_T)
+\int^T_te^{-r(\theta-s)}
q(X^{s,x}_{\theta},Y^{s,x}_{\theta})\,d\theta\,|\FF_t\right).
\end{equation}
\end{corollary}

From (\ref{eq3.7}) we get immediately the early exercise premium
representation for $V$. For instance, for American put option,
\begin{align}
\label{eq3.8} V(s,x)&=Ee^{-r(T-s)}g(X^{s,x}_T)\\
&\quad +E\int^T_s
e^{-r(t-s)}(rK-dX^{s,x}_t)^+\,{\mathbf{1}}_{\{V=g\}}(t,X^{s,x}_t)\,dt.
\nonumber
\end{align}

Representations (\ref{eq3.7}), (\ref{eq3.8}) are known (see
\cite[Corollary 2.7.11]{KS}). To our knowledge our proof is new.
Let us stress, however, that we were influenced by results of
\cite{BKR1}.

\subsection*{Acknowledgements}

Research supported by Polish Ministry of Science and Higher
Education (grant no. N N201 372 436).
\medskip\\
Tomasz Klimsiak, Andrzej Rozkosz\\
Faculty of Mathematics and Computer Science\\
Nicolaus Copernicus University\\
Chopina 12/18\\
87-100 Toru\'n, Poland\\
E-mail: tomas@mat.uni.torun.pl, rozkosz@mat.uni.torun.pl


\begin{thebibliography}{10}
\bibitem{BM}
V. Bally,  and A. Matoussi, {\em Weak solutions for SPDEs and
backward doubly stochastic differential equations}, J. Theoret.
Probab. 14 (2001), 125--164.

\bibitem{BKR1}
F.S. Benth, K.H. Karlsen, and K. Reikvam, {\em A semilinear Black
and Scholes partial differential equation for valuing American
options},  Finance Stoch.  7 (2003), 277--298.

\bibitem{BKR2}
F.S. Benth, K.H. Karlsen, and K. Reikvam,  {\em On a semilinear
Black and Scholes partial differential equation for valuing
American options. Part II: approximate solutions and convergence},
Interfaces Free Bound.  6 (2004), 379--404.

\bibitem{EKPPQ}
N. El Karoui, C. Kapoudjian, E. Pardoux, S. Peng and M.C. Quenez,
{\em  Reflected solutions of backward SDEs, and related obstacle
problems for PDE's}, Ann. Probab. 25 (1997), 702--737.

\bibitem{EPQ}
N. El Karoui, S. Peng and M.C. Quenez,  {\em Backward Stochastic
Differential Equations in Finance}, Mathematical Finance 7 (1997),
1--77.

\bibitem{EQ}
N. El Karoui and M.C. Quenez, {\em Non-linear pricing theory and
backward stochastic differential equations}, Lecture Notes in
Math. 1656 (1997), 191--246.

\bibitem{Friedman}
A. Friedman, {\em Partial Differential Equations of Parabolic
Type}, Prentice-Hall, Englewood Cliffs, N.J., 1994.

\bibitem{KS}
I. Karatzas and S.E. Shreve {\em Methods of Mathematical Finance},
Springer, New York, 1998.

\bibitem{K}
T. Klimsiak, {\em On Time-Dependent Functionals of Diffusions
Corresponding to Divergence Form Operators}, J. Theoret. Probab.
(2011), DOI 10.1007/s10959-011-0381-4.

\bibitem{Lad}
O.A. Ladyzenskaya, V.A. Solonnikov and N.N. Ural'ceva,  {\em
Linear and Quasi-Linear Equations of Parabolic Type}. Transl.
Math. Monographs {\bf 23}. Amer. Math. Soc., Providence, R.I.,
1968.

\bibitem{Nualart}
D. Nualart, {\em The Malliavin Calculus and Related Topics}.
Springer, Berlin, 1995.

\bibitem{Pa}
\'E. Pardoux, {\em Backward Stochastic Differential Equations and
Viscosity Solutions of Systems of Semilinear Parabolic and
Elliptic PDEs of Second Order}, in: {\em Stochastic Analysis and
Related Topics VI} (The Geilo Workshop, 1996),  L. Decreusefond,
J. Gjerde, B. \O ksendal, A.S. \"Ust\"unel (eds.), Birkh\"user,
Boston, 1998, 79--127.

\end{thebibliography}
\end{document}